\documentclass[11pt]{article}
\usepackage[margin=25mm]{geometry}

\usepackage{tabularx}
\usepackage{indentfirst}
\usepackage{setspace}

\usepackage{url}
\usepackage{color}
\usepackage{mathrsfs}

\usepackage[pdftex]{graphicx}

\usepackage{tikz-cd}
\usepackage{amscd}
\usepackage{enumerate}
\usepackage{mathtools}
\usepackage{amsmath}
\usepackage{amssymb}
\usepackage{amsthm}
\usepackage{usebib}
\usepackage{comment}
\usepackage{graphics}

\newcommand{\abs}[1]{\left|#1\right|}

\newcommand{\N}{\mathbb{N}}
\newcommand{\Z}{\mathbb{Z}}

\newcommand{\R}{\mathbb{R}}
\newcommand{\C}{\mathbb{C}}
\newcommand{\F}{\mathbb{F}}

\newcommand{\CC}{\mathcal{C}}

\newcommand{\sh}{\mathcal{S}}

\newcommand{\Hom}{\mathrm{Hom}}

\newcommand{\Image}{\mathrm{Im}}

\newcommand{\Id}{\mathrm{id}}

\newcommand{\col}{\mathrm{Col}}

\newcommand{\Para}{\mathrm{Para}}

\newcommand{\SL}{\mathrm{SL}}
\newcommand{\GL}{\mathrm{GL}}

\newcommand{\cov}{\mathrm{cov}}

\newcommand{\DW}{\mathrm{DW}}

\newcommand{\restr}[2]{{
  \left.\kern-\nulldelimiterspace
  #1
  \vphantom{\big|}
  \right|_{#2}
  }}

\theoremstyle{plain}
\newtheorem{Theorem}{Theorem}[section]
\newtheorem{Proposition}[Theorem]{Proposition}
\newtheorem{Lemma}[Theorem]{Lemma}
\newtheorem{Corollary}[Theorem]{Corollary}

\theoremstyle{definition}
\newtheorem{Definition}[Theorem]{Definition}

\newtheorem{Remark}[Theorem]{Remark}

\begin{document}

\title{Parabolic Dijkgraaf-Witten invariants of links in the 3-sphere}
\author{
	Koki Yanagida 
	\thanks{
	\text{Department of Mathematics, Tokyo Institute of Technology, 2-12-1 Ookayama, Meguro-ku, Tokyo 152-8551, JAPAN}
		\text{\quad \quad e-mail:\texttt{yanagida.k.ab@m.titech.ac.jp}
		}
	}
}

\sloppy

\date{\empty}

\maketitle



\begin{abstract}
We define a new invariant of links in the $3$-sphere and call it the parabolic Dijkgraaf-Witten (DW) invariant. This invariant is a generalization of the reduced DW invariant derived by Karuo. In this paper, we compute the invariant of several links over which double branched coverings are homeomorphic to the lens spaces. Moreover, we introduce a procedure for computing partial information of the parabolic DW invariant using only link diagrams.
\end{abstract}

\begin{center}
\normalsize
\baselineskip=11pt
{\bf Keywords} \\
Dijkgraaf--Witten invariant, link invariant \ \ \
\end{center}
\begin{center}
\normalsize
\baselineskip=11pt
{\bf Subject Code } \\
\ \ 	57K10
\end{center}
\large 

\section{Introduction}
For a finite group $G$, the Dijkgraaf-Witten (DW) invariant \cite{DW} is a topological invariant of closed 3-manifolds, and is simply defined as an element in the group ring $\Z [ H_3 (G) ]$.
Here, $H_3 (G)$ is the group homology of $G$.
Dijkgraaf and Witten established a method of computing the invariant from a triangulation of $M$. 
However, when $G$ is non-abelian, it is not easy to compute the DW invariant because it is generally difficult to determine the group structure of $H_3 (G)$ explicitly. 

There have been several studies on generalizations of the DW invariant applicable to compact oriented $3$-manifolds with boundaries \cite{Kim18, Wakui}.
In one such study, for the case in which $G=\SL_2 (\C)$, Zickert \cite{Zic} introduced the relative group homology and defined an invariant of 3-manifolds with boundaries.
Analogously, when $G=\SL_2(\F_p)$ for some prime $p$, Karuo \cite{Kar212, Kar1} defined a reduction of the DW invariant as an invariant of knots in $S^3$.
He suggested a method to compute the reduced DW invariants via ideal triangulations of knot complements; however, to ensure that Karuo's approach works, many aspects of the method must be confirmed.

In this paper, we focus on the case $G=\SL_2 (\F_q)$, and define the parabolic DW invariant as an invariant of links.
Here, $q$ is a power of a prime $p$.
The point is that the parabolic DW invariant is defined as an analogue of the DW invariant of $B_L^p$, where $B_L^p$ is the $p$-fold cyclic covering space of $S^3$ branched over a link $L$.
Although the definition of the parabolic DW invariant is relatively simple, it turns out to be a generalization of the above reduced DW invariant defined by Karuo (Proposition \ref{prop8}).
Moreover, we compute the parabolic DW invariants of several links whose double branched covering spaces are homeomorphic to some lens spaces (Theorem \ref{hukugen}).
As an application, we compute parabolic DW invariants of the $(2,2m)$-torus link and the $2m$-twist knot for all $m \in \N$. 
Appendix \ref{app2} contains the computations for $1 \leq m \leq 50$. (Tables \ref{paraDW_of_torus} and \ref{paraDW_of_twist}).

Moreover, using quandle theory \cite{Nos17} and techniques relating to branched coverings, we establish a procedure for recovering the reduced DW invariant from the parabolic DW invariant (see Theorem \ref{prop888} and Proposition \ref{redDW}).
This procedure is based on some quandle colorings of a link diagram. 
Therefore, compared with Karuo's computation \cite{Kar212, Kar1} via ideal triangulations,  our procedure allows easier computations of the reduced DW invariant.
As a corollary, when $7 \leq p \leq 37$, we compute the reduced DW invariants of the prime knots with seven or fewer crossings (Table \ref{thepairing}).

The remainder of this paper is organized as follows.
Section \ref{sec2} examines the parabolic DW invariant.
In Section \ref{ss5}, we introduce a computation of the reduced DW invariant with quandle cocycle invariants and present some examples of this computation.
In Appendix \ref{app1}, we present another definition of the parabolic DW invariant and show that it is a generalization of the reduced DW invariant \cite{Kar212, Kar1}.
Appendix \ref{app2} provides specific values of the parabolic DW invariants for the $(2,2m)$-torus link and the $2m$-twist knot with $1 \leq m \leq 50$.

\noindent{\bf Conventional terminology.}
Throughout this paper, we fix a prime $p \in \N$ and denote the field of order $q$ as $\F_{q}$, where $q$ is a power of $p$.
We do not specify the coefficient group when it is $\Z$.

\section{Parabolic Dijkgraaf-Witten invariant of links}\label{sec2}
In this section, we define the parabolic DW invariant (Definition \ref{maindef}) and prove some properties;
we establish a decomposition of the parabolic DW invariant (Theorem \ref{hukugen}),
and demonstrate an invariance property of the parabolic DW invariant (Lemma \ref{concor} and Corollary \ref{periodicity}).

\subsection{Definition of parabolic Dijkgraaf-Witten invariant}
Before defining the parabolic DW invariant, we begin by reviewing the DW invariant \cite{DW}.
Let $M$ be an oriented closed 3-manifold, and let $[M] \in H_3(M;\Z)$ denote the fundamental class of $M$.
For a group $G$, let $BG$ be an Eilenberg-MacLane space of type $(G,1)$.
Fix a classifying map $\iota: M \to B\pi_1(M)$.
For a group homomorphism $f : \pi_1(M) \to G$, we write $f_*: H_3(M) \to H_3(BG)$ for the composite map $ (Bf)_* \circ \iota_*$.
Here, $Bf : B\pi_1(M) \to BG$ is a continuous map induced by $f$.
Note that $f_*$ is well-defined since $Bf$ and $\iota$ are unique up to homotopy.
Then, the {\it DW invariant} of $M$ \cite{DW} is defined to be the formal sum
\begin{equation}\label{oriDW}
\mathrm{DW}(M) \coloneqq \sum_{f \in \Hom(\pi_1(M),G)} 1_{\Z} \, f_*( [ M]) \in \Z [H_3(BG)] ,
\end{equation}
where $ \Z [H_3(BG)] $ is the group ring of $H_3(BG) $ over $\Z$ and $1_\Z$ is the unit of the group ring.
It is well-known that the group homology of $G$ is isomorphic to the ordinary homology of $BG$.
As stated in the introduction, it is not easy to generalize the DW invariant to $3$-manifolds with boundaries (see, e.g., \cite{Wakui}).

Next, we define the parabolic DW invariant.
We consider the case where $G= \mathrm{SL}_2 (\F_q) $, and discuss link complements with toroidal boundaries.
Let $L \subset S^3$ be a link and $f : \pi_1(S^3 \setminus L ) \to \mathrm{SL}_2 (\F_q)$ be a representation.
The representation $f$ is called {\it parabolic} if the image of each meridian of $L$ under $f$ is conjugate to $ \begin{pmatrix}
1 & * \\
0 & 1 \\
\end{pmatrix} $ for some $* \in \F_q^{\times}$.
Write $\Para(L, q)$ for the set of parabolic representations. 
That is,
\[
\Para(L, q) \coloneqq
\{ f \in \Hom (\pi_1(S^3 \setminus L ),\SL_2 (\F_q)) \mid f \text{ is parabolic} \}.
\]

Let $\cov : E_L^p \to S^3 \setminus L$ be the $p$-fold cyclic covering, and
let $B_L^p \to S^3$ be the $p$-fold cyclic covering of $S^3$ branched over $L$.
Since the induced map $\cov_* : \pi_1 (E_L^p) \to \pi_1 (S^3 \setminus L)$ is injective, we regard $\pi_1 (E_L^p)$ as a subgroup of $\pi_1 (S^3 \setminus L)$.
Now, consider a parabolic representation $f \in \mathrm{Para}(L, q)$.
From the definition of the parabolic representation, it follows that the restriction $\restr{f}{\pi_1(E_L^p)}$ sends each meridian to the identity matrix.
Thus, $\restr{f}{\pi_1(E_L^p)}$ induces $ \bar{f}: \pi_1(B_L^p) \to \mathrm{SL}_2 (\F_q) $ such that the following diagram commutes:
\[\raisebox{-0.5\height}{\includegraphics{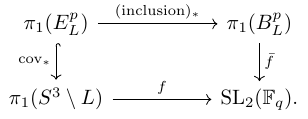}}\]

\begin{Definition}\label{maindef}
Let $L \subset S^3$ be a link.
We define the \emph{parabolic DW invariant} of $L$ to be the formal sum
\begin{equation}\label{def1}
\mathrm{DW}_q(L) \coloneqq \sum_{f \in \Para(L,q)} 1_{\Z} \, \bar{f}_*( [ B_L^p]) \in \Z [H_3(B\mathrm{SL}_2 (\F_q))] .
\end{equation}
\end{Definition}

By definition, $\mathrm{DW}_q$ is a link invariant; we will show that this invariant is a generalization of the reduced DW invariant of Karuo \cite{Kar212, Kar1} in Proposition \ref{prop8}.

\begin{Remark}\label{keyrem}
\begin{enumerate}
\item A classical result says that $H_3(B\SL_2 (\F_q) ; \Z [1/p]) \cong \Z/(q^2-1)$; see, e.g., \cite{Hut13}.
Furthermore, if $q \neq 2, 2^2, 2^3, 3, 3^2, 3^3 ,5$, then the $p$-torsion of $H_3(B\mathrm{SL}_2(\F_q) )$ vanishes.
In this paper, we say that a prime power $q$ is \emph{generic} if $q \neq 2, 2^2, 2^3, 3, 3^2, 3^3 ,5$.

\item If two parabolic representations $f,f': \pi_1(S^3 \setminus L ) \to \mathrm{SL}_2 (\F_q) $ are conjugate,
then $ \bar{f}_*( [ B_L^p]) = \bar{f'}_*( [ B_L^p]) $ since the homomorphism induced by conjugation maps on $\mathrm{SL}_2 (\F_q)$ are the identity maps on homology.

\item For $i \in \{ 1,2 \}$, let $M_i$ be a closed oriented 3-manifold and let $f_i : \pi_1(M_i) \to G$ be a homomorphism.
Consider an oriented connected compact 4-manifold $W$ and a homomorphism $F: \pi_1(W) \to G $ such that $ \partial W= M_1 \sqcup -M_2$ and $f_j =F\circ (\iota_j)_*$.
Here, $\iota_j : M_j \hookrightarrow W$ is the inclusion.
By the long homology exact sequences from $(W, M_1 \sqcup M_2)$, one can show that $(f_1)_* ([M_1]) = (f_2)_* ([M_2]) \in H_3(G)$.
\end{enumerate}
\end{Remark}

\subsection{Computation of partial sum with $p=2$}\label{commet}
In general, it is difficult to evaluate the parabolic DW invariant because the isomorphism $H_3(B\mathrm{SL}_2 (\F_q)) \cong \Z/(q^2-1)$ in Remark \ref{keyrem} is not canonical for every generic prime power $q$.
In this subsection, 
we consider the case where the branched covering space $B_L^p$ is homeomorphic to a lens space.
In this setting, we obtain a decomposition of the parabolic DW invariant into the sum of computable invariants (Theorem \ref{hukugen}).

To state Theorem \ref{hukugen}, we begin by recalling cyclic subgroups $T, K, U \subset \SL_2 (\F_q)$ from \cite{Hut13}.
$T$ is defined as the subgroup consisting of diagonal matrices in $\SL_2 (\F_q)$, which is isomorphic to $\Z / (q-1)$.
By an additive isomorphism $\kappa : \F_{q^2} \cong \F_{q} \oplus \F_q$, we obtain the canonical homomorphism 
\begin{equation}\label{can}
\mu : \F_{q^2}^{\times} \longrightarrow \GL_2(\F_q); \quad a \longmapsto \mu_a,
\end{equation}
where $\mu_a (x) \coloneqq \kappa (a \kappa^{-1} (x))$ for $x \in \F_{q} \oplus \F_q$.
We define $K\subset\SL_2(\F_q)$ by $K \coloneqq \Image \mu \cap \SL_2 (\F_q)$, which is isomorphic to $\Z/(q+1)$.
Let $U = \left\{ \begin{pmatrix} 1 & x \\ 0 & 1 \end{pmatrix} \in \SL_2(\F_q) \middle| x \in \F_q \right\}$, of which order is $q$.

For a subgroup $H \subset \SL_2 (\F_q)$, we define a subset $P_H \subset \Para(L,q)$ by 
\begin{equation*}
P_H \coloneqq \left\{ f \in \Para(L,q) \mid \Image \bar{f} \neq \{ 1 \} \text{ and } \Image \bar{f} \subset g^{-1} H g \text{ for some }g \in \SL_2 (\F_q) \right\}. \\
\end{equation*}
In addition, we define $\mathrm{DW}^H_q(L)$ by a partial sum of \eqref{def1},
\[ \mathrm{DW}^H_q(L) \coloneqq \sum_{f \in P_H} 1_{\Z} \, \bar{f}_*( [ B_L^p ]) \in \Z [H_3( \SL_2 (\F_q))]. \]

Let us consider the situation in which $B_L^2$ is homeomorphic to $L(m,n)$, where $m$ and $n$ are relatively prime integers and $L(m,n)$ is the lens space.

\begin{Theorem}\label{hukugen}
Let $p=2$, $q \geq 16$ and $P_{\Id} \coloneqq \left\{ f \in \Para(L,q) \mid \Image \bar{f} = \{ I_2 \} \right\}$,  where $I_2 \in \SL_2 (\F_q)$ is the identity matrix.
Suppose that the double branched covering space $B_L^2$ is homeomorphic to a lens space $L(m,n)$ for some $m, \, n \in \N$.
Then,
\[
\DW_q (L)=
\mathrm{DW}^T_q(L)+
\mathrm{DW}^K_q(L)+
(\abs{P_U}+\abs{P_{\Id}}) \, 1_{\Z}.
\]
\end{Theorem}
\begin{proof}
We first show that 
\begin{equation}\label{pf2.3}
\Para (L,q) = P_{\Id} \sqcup P_T \sqcup P_K \sqcup P_U.
\end{equation}
Note that it follows immediately from definition that $P_{\Id} \cap P_T = P_{\Id} \cap P_K = P_{\Id} \cap P_U = \emptyset$.
Suppose that $\rho \in \Para (L,q) \setminus P_{\Id}$, and
let
$g \coloneqq
\begin{pmatrix}
a & b \\
c & d
\end{pmatrix}
\in \SL_2 (\F_q) \setminus \{ I_2\}$
be the image of a generator of $\pi_1 (L(m,n)) \cong \Z/m$ under $\bar{\rho}$.
Let $p_g(t)$ denote the characteristic polynomial of $g$, namely $\det(tI_2 - g)$.
Then, $p_g(t)$ is either irreducible or equal to $(t-\alpha)(t-\alpha^{-1})$ for some $\alpha \in \F_q^{\times}$.
Hence, in order to derive (\ref{pf2.3}), it is enough to show the following three claims:

\begin{description}
\item[{\bf Claim (i):} $p_g(t)=(t-\alpha)(t-\alpha^{-1})$ for some $\alpha \neq 1$ if and only if $\rho \in P_T$.]\mbox{}\\
{\it Proof of Claim (i).}
The ``only if'' part is obvious.
Conversely, suppose that $p_g(t)=(t-\alpha)(t-\alpha^{-1})$ for some $\alpha \neq 1$.
Then, $g$ is conjugate to
$\begin{pmatrix}
\alpha & 0 \\
0 & \alpha^{-1}
\end{pmatrix}$.
Thus, $\rho \in P_T$.

\item[{\bf Claim (ii):} $p_g(t)=(t-1)^2$ if and only if $\rho \in P_U$.]\mbox{}\\
{\it Proof of Claim (ii).}
The ``only if'' part is obvious.
Conversely, suppose that $p_g(t)=(t-1)^2$.
Because $\overline{\rho}$ is non trivial, $g$ must have the Jordan normal form
$\begin{pmatrix}
1 & 1 \\
0 & 1
\end{pmatrix},$ which implies that $\rho \in P_U$.

\item[{\bf Claim (iii):} $p_g(t)$ is irreducible if and only if $\rho \in P_K$.]\mbox{}\\
{\it Proof of Claim (iii).}
We first show ``if'' part.
Suppose that $p_g(t)$ is irreducible.
Then, either $b$ or $c$ must be non zero.
Under an identification $\F_{q^2} \cong \F_q[X] / (p_{g}(X))$, direct computations show
$\mu_X =
\begin{pmatrix}
0 & -1 \\
1 & a+d
\end{pmatrix} 
\in \SL_2 (\F_q)$.
Here, $\mu : \F_q[X] / (p_{g}(X)) \to \GL_2 (\F_q)$ is the canonical map in (\ref{can}).
Define $h \in \SL_2 (\F_q)$ by
\begin{equation*}
h =
\left\{
\begin{array}{ll}
\left(
\begin{array}{cc}
\sqrt{c} & d \sqrt{c}^{-1} \\
(a + d) \sqrt{c} & b \sqrt{c} + d^2 \sqrt{c}^{-1} \\
\end{array}
\right), &\text{if } c \neq 0, \\ \\
\left(
\begin{array}{cc}
a \sqrt{b}^{-1} & \sqrt{b} \\
a^2 \sqrt{b}^{-1} + c \sqrt{b} & (a+d) \sqrt{b} \\
\end{array}
\right), &\text{if } c = 0.
\end{array}
\right.
\end{equation*}
Here, $\sqrt{\bullet} : \F_q \to \F_q$ is the inverse map of the Frobenius endomorphism $\F_q \to \F_q$, which sends each element to its square.
Then, $h g h^{-1} = \mu_X$, which leads to $\rho \in P_K$.

Conversely, suppose that $\rho \in P_K$ and $p_g(t)$ is reducible.
Then, from $\rho \in P_K$, it follows that the order of $g$ is a divisor of $p+1$.
In addtion, since $p_g(t)$ is reducible, Claim (i) and (ii) implies the order of $g$ is also a divisor of $p$ or $p-1$.
Because $p+1$ is relatively prime to $(p-1)p$, we have $g = I_2$.
This contradicts the assumption that $\rho \notin P_{\Id}$; hence, $p_g(t)$ is irreducible if $\rho \in P_K$.
\end{description}
Hence, (\ref{pf2.3}) holds.
Therefore,
\allowdisplaybreaks
\begin{align*}
\mathrm{DW}_q(L)
&=	\sum_{f \in \Para(L,q)} 1_{\Z} \, \bar{f}_*( [ B_L^p]) \\
&=	\sum_{f \in P_T} 1_{\Z} \, \bar{f}_*( [ B_L^p])
+\sum_{f \in P_K} 1_{\Z} \, \bar{f}_*( [ B_L^p])
+\sum_{f \in P_U} 1_{\Z} \, \bar{f}_*( [ B_L^p])
+\sum_{f \in P_{\Id}} 1_{\Z} \, \bar{f}_*( [ B_L^p]) \\
&=	 \mathrm{DW}^T_q(L) 
+ \mathrm{DW}^K_q(L) 
+(\abs{P_U}+\abs{P_{\Id}}) \, 1_{\Z},
\end{align*}
where the third equality is obtained by $H_3 (U) \cong \Z/q$ and $H_3 (\SL_2 (\F_q) ; \Z_2) \cong 0$ \cite{Hut13}.
\end{proof}
Theorem \ref{hukugen} implies that, in order to compute $\DW_q (L)$ for $q=2^j$ with $j \geq 4$, it suffices to compute $\DW^T_q (L)$ and $\DW^K_q (L)$.
For the remainder of this subsection, we assume $H = T$ or $K$, and discuss a computational method of $\DW^H_q (L)$.

For $f \in P_H$, it follows from Remark \ref{keyrem} that $\overline{f}_* ([B_L^2]) \in H_3(H)\subset H_3(\SL_2(\F_p))$.
Therefore, we may regard $\DW^H_q (L)$ as an element in $\Z[H_3(H)]$, and need only consider the pushforward $\overline{f}_* ([B_L^2]) = (B\overline{f})_* \circ \iota_* ([B_L^2])$ by $\overline{f} : \pi_1(B_L^2) \cong \Z/m \to H$.
We first observe $\iota: B_L^2 \to B\Z/m$.
It is well-known that $B\Z/m$ is an infinite lens space $L^\infty (m)$ of which $3$-skeleton is isomorphic to $L(m,n)$.
Furthermore, we can take $\iota$ to be the inclusion map $B_L^2 \cong L(m,n) \to B\Z/m$.
Hence, we may assume that $\iota_*: H_3(B_L^2) \to H_3(B\Z/m)$ coincides with the quotient map $\Z \to \Z/m$.
In addition, one can directly compute $Bf_* : H_3(\Z/m) \to H_3(H)$ from the definition of the group homology.

In conclusion, we can easily compute $\DW^H_q (L)$.
In fact, Tables \ref{paraDW_of_torus} and \ref{paraDW_of_twist} in Appendix B present two specific examples.

\subsection{Invariance of parabolic Dijkgraaf-Witten invariant}
In this subsection, we first prove that the parabolic DW invariant is invariant under a certain concordant operation:
\begin{Lemma}\label{concor}
As shown in Fig.\ref{L1L2L3}, consider three link diagrams with arcs $\alpha, \beta, \gamma, \delta, \alpha', \beta', \gamma'$ and $\delta'$.
For any parabolic representation $f \in \Para(L_1,q)$ such that $ f(\alpha)=f(\gamma)$ and $f(\beta)=f(\delta) $,
there are parabolic representations $f_j \in \Para(L_j,q)$ with $j \in \{ 2, 3 \} $ such that $(\overline{f_2})_* ( [B^p_{L_2} ]) + (\overline{f_3})_* ( [B^p_{L_3} ]) = (\overline{f})_* ( [B^p_{L_1} ])$.
\begin{figure}[htbp]
	\begin{center}
		\includegraphics[scale=0.6]{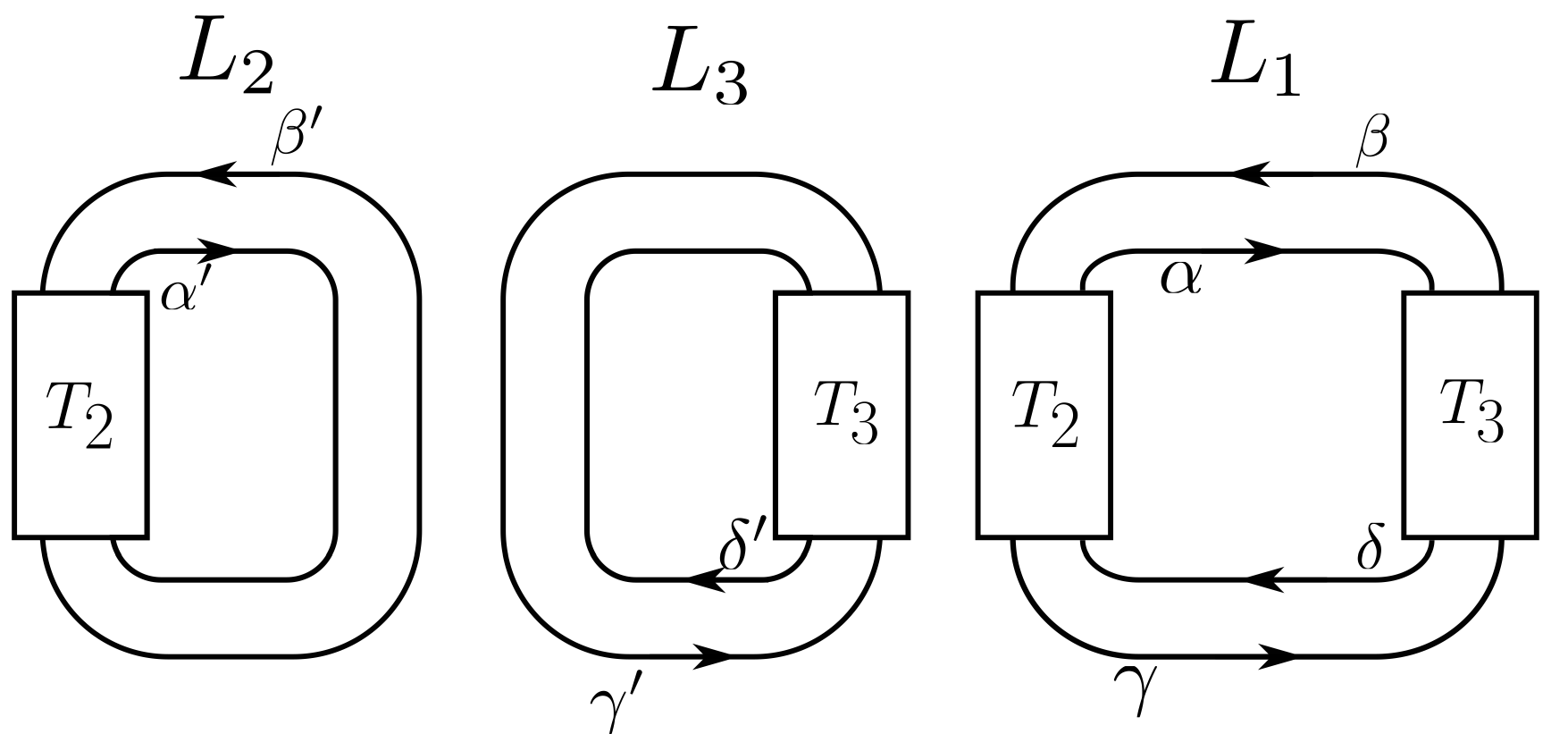}
		\caption{The links $L_1, L_2$, and $L_3$.\label{L1L2L3}}
	\end{center}
\end{figure}
\end{Lemma}
\begin{proof}
Define $f_2 \in \Para(L_2,q)$ and $f_3 \in \Para(L_3,q)$ by $f(\alpha)= f_2(\alpha' )=f_3(\gamma' )$ and $f(\beta)= f_2(\beta' )=f_3(\delta' )$ canonically.
We consider the bands $\mathcal{D} \coloneqq \mathcal{D}_1 \cup \mathcal{D}_2$ that connect $L_2$ and $L_3$ (see Fig.\ref{DDL2L3}).
\begin{figure}[htbp]
\begin{center}
\includegraphics[scale=0.4]{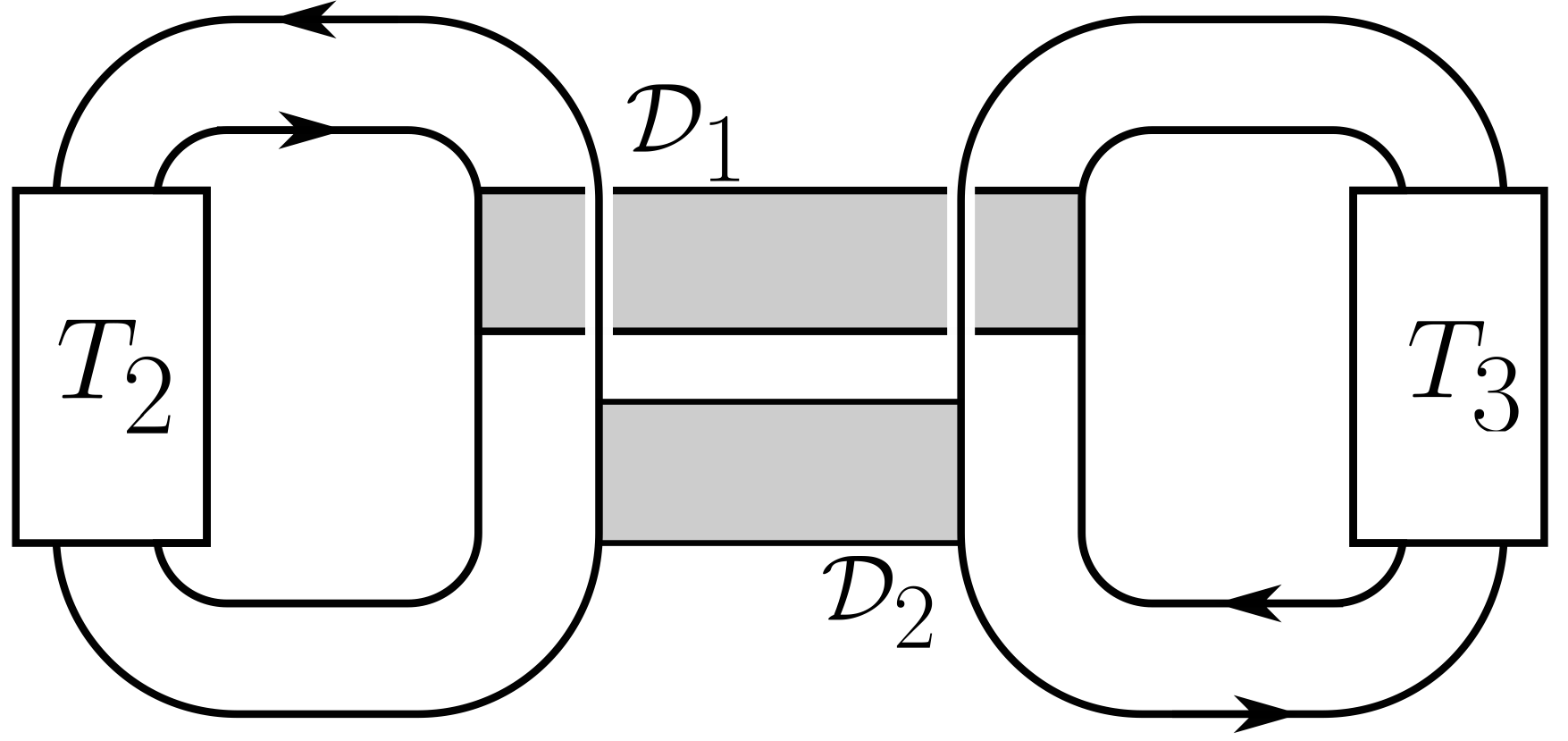}
\vspace{-10pt}
\caption{Bands $\mathcal{D}$ with $L_2$ and $L_3$.\label{DDL2L3}}
\end{center}
\end{figure}

Regard the bands as two oriented disks $\mathcal{D}$ in the 4-ball $B^4$ such that $\partial \mathcal{D} \cap S^3 =\mathcal{D} \cap (L_2 \sqcup - L_3) $.
Let $ B_\mathcal{D} \to B^4$ denote the $p$-fold cyclic covering branched over $\mathcal{D}$.
By assumption, we have a representation $F: \pi_1( B^4 \setminus \mathcal{D}) \to \mathrm{SL}_2(\F_q)$, which
induces $\bar{F}: \pi_1 ( B_\mathcal{D} ) \to \mathrm{SL}_2(\F_q)$ such that $\partial B_\mathcal{D} = B_{L_2}^p \sharp B_{L_3}^p$ and $\bar{f}_2 *\bar{f}_3 =\bar{F}\circ \iota $.
Note the homeomorphism $B_{L_2}^p \sharp B_{L_3}^p \sharp (S^2 \times S^1) \cong B_{L_2 \sqcup L_3}^p $ given by the construction of branched covering spaces.
Hence, by Remark \ref{keyrem}, we obtain the required equality as follows:
\[ \overline{(f_2)}_*( [ B_{L_2}^p]) +\overline{(f_3)}_*( [ B_{L_3}^p]) = \overline{(f_2*f_3)}_*( [ B_{L_2}^p \sharp B_{L_3}^p \sharp S^2 \times S^1 ]) = \overline{( f_2 \sqcup f_3)}_*( [ B_{ L_2 \sqcup L_3}^p]) =
\overline{f }_*( [ B_{L_1}^p]). \]
\end{proof}
For $m \in \N$, let $\mathcal{T}_m$ be the $(2,2m)$-torus link and $\mathcal{K}_m$ be the $2m$-twist knot (Fig.\ref{TmKm}).
As a corollary,
\begin{figure}[htbp].
\begin{center}
\includegraphics[width=120mm]{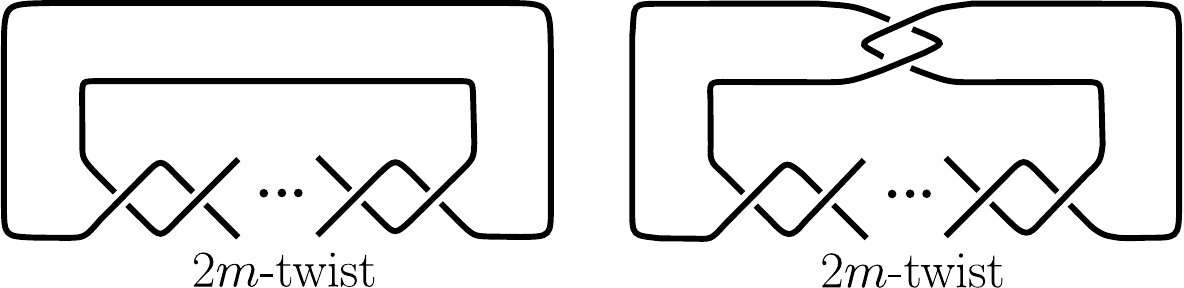}
\caption{$(2,2m)$-torus link $\mathcal{T}_m$ and $2m$-twist knot $\mathcal{K}_m$.\label{TmKm}}
\end{center}
\end{figure}
\begin{Corollary}\label{periodicity}
Let $q$ be generic and $\lambda = (q^3-q)(q^2-1)$. Then, for any $m\in\N$,
\[\DW_q (\mathcal{T}_{m+\lambda}) = \DW_q (\mathcal{T}_{m}) , \qquad \DW_q (\mathcal{K}_{m+\lambda}) = \DW_q (\mathcal{K}_{m}).\]
\end{Corollary}
\begin{proof}
We focus on the proof for the torus links $\mathcal{T}_m$, 
because the latter claim for the twist knots $\mathcal{K}_m$ follows in the same manner as the proof of the former claim for the torus links $\mathcal{T}_m$.

Considering a diagram of the torus link $\mathcal{T}_m$ with arcs $Z_1, \ldots, Z_{2m+2}$, as shown in Fig.\ref{WirTm}, we obtain the following Wirtinger presentation:
\begin{equation}\label{hooh}
\pi_1 (S^3 \setminus \mathcal{T}_m) = \left\langle Z_1, \ldots, Z_{2m+2} \middle|
\begin{array}{l}
Z_{2m+1}=Z_1, \,\,\,\,\,\,\,
Z_{2m+2}=Z_2, \\
Z_{2k+1} = (Z_2^{-1} Z_1)^k Z_1 (Z_2^{-1} Z_1)^{-k}, \\
Z_{2k+2} = (Z_2^{-1} Z_1)^k Z_2 (Z_2^{-1} Z_1)^{-k}
\end{array}
\, (k=1, \ldots, m)
\right\rangle.
\end{equation}
\begin{figure}[htbp].
\begin{center}
\includegraphics[width=100mm]{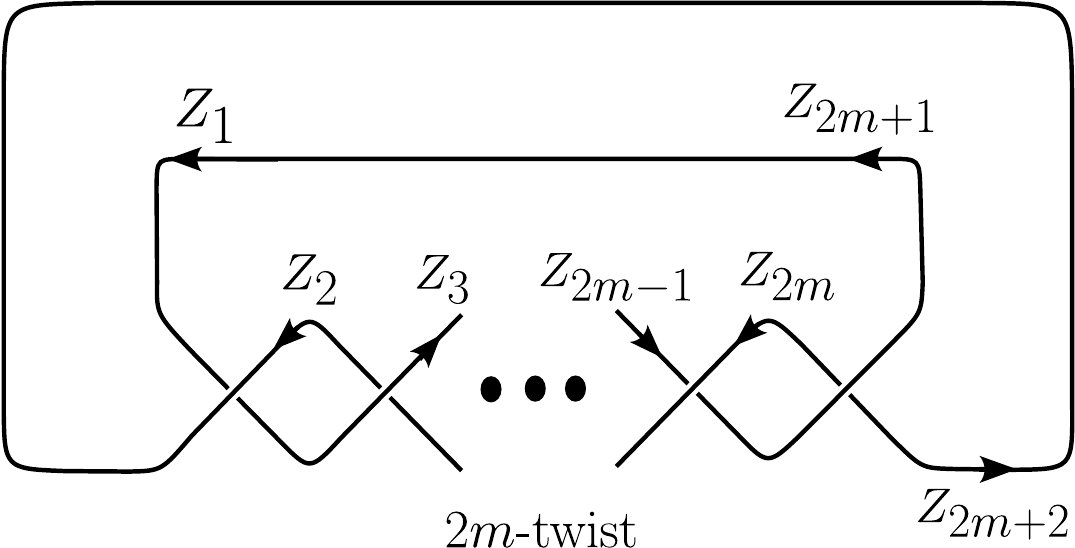}
\caption{Arcs of $\mathcal{T}_m$.}
\label{WirTm}
\end{center}
\end{figure}
If we replace $m$ in (\ref{hooh}) by $m+\lambda$, then we also get the Wirtinger presentation of $\pi_1 (S^3 \setminus \mathcal{T}_{m+\lambda})$.

Let $f : \pi_1 (S^3 \setminus \mathcal{T}_{m+\lambda})\to \SL_2 (\F_q)$ be a parabolic representation.
We have
\[
\left\{
\begin{array}{ll}
f(Z_1) &= f(Z_{2(q^3-q)+1}) = f(Z_{4(q^3-q)+1}) = \cdots = f(Z_{2\lambda+1}), \\
f(Z_2) &= f(Z_{2(q^3-q)+2}) = f(Z_{4(q^3-q)+2}) = \cdots = f(Z_{2\lambda+2 }),
\end{array}
\right.
\]
since $| \SL_2 (\F_q) | = q^3-q$. 
Thus, by repeatedly applying Lemma \ref{concor}, there are parabolic representations $\mathfrak{f} \in \Para(\mathcal{T}_m, q)$ and $f' \in \Para(\mathcal{T}_{ q^3-q }, q)$ such that 
$(\overline{f})_* ( [B^p_{\mathcal{T}_{m+\lambda}} ]) = (\overline{\mathfrak{f}})_* ( [B^p_{\mathcal{T}_m} ]) + (q^2-1) (\overline{f'})_* ( [B^p_{\mathcal{T}_{ q^3-q }} ])$.
Because $| H_3 (\SL_2 (\F_q) ; \Z)| = q^2-1$ \cite{Hut13}, it follows that
\[
(\overline{f})_* ( [B^p_{\mathcal{T}_{m+\lambda}} ]) = (\overline{\mathfrak{f}})_* ( [B^p_{\mathcal{T}_m} ]).
\]

Recalling the proof of Lemma \ref{concor}, one can easily verify that the map $\Para(\mathcal{T}_{m+\lambda}, q) \to \Para(\mathcal{T}_m, q)$, which  sends $f$ to $\mathfrak{f}$, is bijective.
Hence, $\DW_q (\mathcal{T}_{m+\lambda}) = \DW_q (\mathcal{T}_{m})$.
\end{proof}

\section{Computation from the quandle cocycle invariant}\label{ss5}
In this section, we describe a procedure for computing a reduction of the parabolic DW invariant \eqref{def1} in terms of quandle cocycle invariants.
For this, we review the cocycle invariants to compute the reduction (see \eqref{pp24} below).
In addition, we compute the reduced invariants of several links; see Section \ref{ss51}.

Throughout this section, we assume that the prime power $q =p^d$ is generic and odd, and fix a non-square number $r_0 \in \F_q$.

\subsection{Review of quandle cocycle invariants }\label{ss50}
First, we review some quandles and colorings. For $r \in \F_q^{\times }$, let $X_r$ be the quotient of $\F_q \times \F_q\setminus\{ (0,0)\}$ subject to
the relation $(x,y)\sim (-x,-y)$ for any $x,y \in \F_q. $ The order of $X_r$ is $(q^2-1)/2$.
Define the binary map $\lhd : X_r \times X_r \to X_r$ by
\[ (a,b) \lhd (c,d) = (a,b) \begin{pmatrix}
1+rcd & rd^2 \\
-rc^2 & 1-rcd \\
\end{pmatrix} .\]
The pair $(X_r, \lhd)$ is sometimes called {\it the parabolic quandle} (see Example 3.15 in \cite{Nos17}).
For a link diagram $D \subset \R^2$ of a link $L$, {\it an $X_r$-coloring} is a map $\mathcal{C}:\{ \textrm{arcs of }D\} \to X_r$ satisfying the condition on the left diagram of Fig.\ref{fig.color} at each crossing of $D$.
We further define a {\it shadow coloring} to be a pair of an $X_r$-coloring $\CC $ and a map $\lambda $ from the complementary regions of $D$ to $X_r$ such that,
if the regions $R$ and $R ' $ are separated by an arc $\alpha$ as shown on the right diagram of Fig.\ref{fig.color}, the equality $\lambda(R ) \lhd \CC (\alpha )= \lambda (R ')$ holds and the unbounded region is assigned by $(1,0) \in X_r$.
Let $\col_{X_r}(D )$ denote the set of shadow colorings of $D$, and let $\col(D)$ be $\col_{X_{1}}(D) \cup \col_{X_{r_0}}(D)$.
Then, there is a bijection (see, e.g., \cite[Example 3.16]{Nos17})
\begin{equation*}\label{eqq1} \mathcal{F}: \col(D) \stackrel{1:1}{\longleftrightarrow} \Para(L,q).
\end{equation*}
\begin{figure}[htpb]
\begin{center}
\includegraphics[width=100mm]{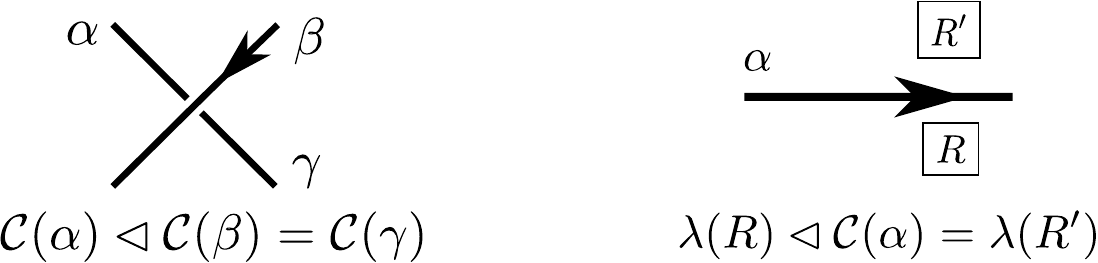}
\end{center}
\caption{Coloring conditions at each crossing and around arcs.
\label{fig.color}}
\end{figure}

Next, we briefly review the quandle (co)-homology.
Let us construct a complex by setting the free $\Z$-module $ C_n^R(X_r) $ spanned by $( x_1, \dots , x_n) \in X_r^n$, 
and let its boundary $\partial_n^{R}( x_1, \dots , x_n) \in C_{n-1}^{R }(X_r)$ be
$$ \sum_{2 \leq i \leq n} (-1)^i\bigl(( x_1, \dots,x_{i-1},x_{i+1},\dots,x_n) - (x_1 \lhd x_i,\dots,x_{i-1}\lhd x_i,x_{i+1},\dots,x_n)\bigr).$$
The composite $\partial_{n-1}^R \circ \partial_n^R $ is known to be zero. The pair $(C_*^R(X_r), \partial^R_*)$ is called the {\it rack complex}.
Let $C^D_n (X_r)$ be a submodule of $C^R_n (X_r)$ generated by $n$-tuples $(x_1, \dots ,x_n)$
with $x_i = x_{i+1}$ for some $ i \in \{1, \dots , n-1\}$.
Because $ \partial_n^R (C^D_n (X_r) ) \subset C^D_{n-1} (X_r)$, we can define a complex $\bigl( C^Q_* (X_r), \partial_* \bigr) $ as the quotient complex $C^R_* (X_r) /C^D_* (X_r)$. The homology $H^Q_n (X_r) $ is called {\it the quandle homology} of $X_r$.
If the prime power $q = p^d$ is generic, then (see \cite[Theorem C.7{}]{Nos17})
\begin{equation*}\label{eqq2} H_2^Q( X_r) \cong (\Z/p)^{d}, \ \ \ \ \ H_3^Q( X_r) \cong \Z/ (q^2-1) \oplus (\Z/p)^{d(d+1)/2}.
\end{equation*}
Dually, for an abelian group $A$, we can define the cohomology groups $H_R^n (X_r;A)$ and $H_Q^n (X_r;A) $ in terms of the coefficient $A$.
For example, any cohomology 3-class in $H_Q^3(X_r;A )$ is represented by a map $\psi: (X_r)^3 \to A$ satisfying
\[ \psi( x ,z,w)-\psi(x,y,w)+\psi(x,y,z)=\psi( x\lhd y,z,w)-\psi(x\lhd z,y\lhd z,w)+\psi( x\lhd w,y\lhd w,z \lhd w ),\]
\begin{equation*}\label{eqq3}
\psi(x, x,y)=\psi(x,y,y)=0,
\end{equation*}
for any $x,y,z,w\in X_r.$
Such a 3-cocycle $\psi$ is called a \emph{quandle 3-cocycle of} $X_r$.

We now briefly review the (shadow) quandle cocycle invariants \cite{CKS01}.
Let $D$ be a diagram of a link $L$ and $\sh \in \col_{X_r}(D) $ be a shadow coloring.
For the crossing $\tau$ shown in Fig.\ref{fig.color43},
we define the {\it weight} of $\tau $ to be $\epsilon_{\tau} (x,y,z) \in C_3^Q(X_r )$, where $\epsilon_{\tau} \in \{ \pm 1\}$ is the sign of $\tau$ according to Fig.\ref{fig.color43}.
Then, {\it the fundamental class of $\sh$} is defined to be $\sum_{\tau}\epsilon_{\tau} (x,y,z) \in C^Q_3(X_r) $; this is a 3-cycle.
We denote the homology class in $ H_3^Q( X_r )$ by $[\sh]$.
As a corollary of \cite{Nos15} and \cite[Corollary 6.20]{Nos17}, there is a surjective homomorphism $\mathcal{I}:H_3^Q( X_r ) \twoheadrightarrow H_3( \mathrm{SL}_2(\F_q))$ such that
\begin{equation}\label{pp24} \mathcal{I}([\sh ]) = \overline{\mathcal{F}(\sh)}_*([ B_L^p])\in H_3( \mathrm{SL}_2(\F_q)) ,\end{equation}
for any $\sh \in \col(D) $. To conclude, the formal sum $\sum_{ \sh \in \col_{X_r}(D) } 1_{\Z} [\sh ] \in \Z[H_3^Q( X_r ) ]$ without $p$-torsion is
equivalent to the parabolic DW invariant of the link $L$ because $p$ annihilates $\mathrm{Ker} \mathcal{I}$.
\begin{figure}[htpb]
\begin{center}
\includegraphics[width=60mm]{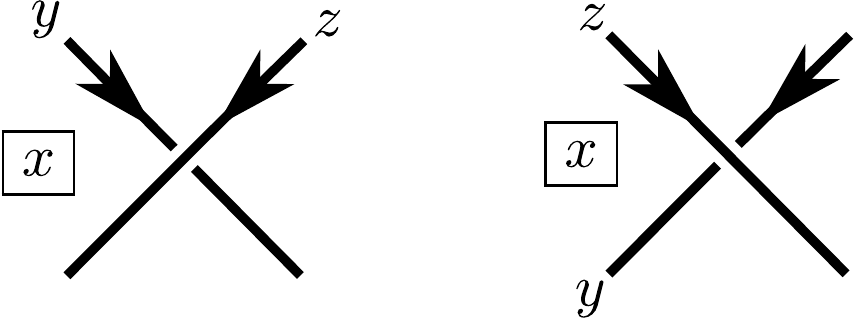}
\end{center}
\vskip -1.737pc
\caption{Positive and negative crossings with $X_r$-colors.
\label{fig.color43}}
\end{figure}

However, it is generally difficult to perform this evaluation in the homology $H_3^Q( X_r ) $.
As usual, a quandle 3-cocycle $\psi \in C^3_Q(X; A) $ and the pairing $\langle \psi, [\sh ]\rangle \in A $ are commonly considered;
this pairing is often called the {\it quandle cocycle invariant} \cite{CKS01}.
In conclusion, to compute the parabolic DW invariants \eqref{def1}, it is crucial to find an explicit formula for $\psi$.

\subsection{Quandle 3-cocycles from the Bloch group }\label{ss51}
In the previous section, using qundle $3$-cocycles, we introduced a procedure for evaluating a reduction of the parabolic DW invariant.
While it is generally difficult to find non trivial 3-cocycles, the purpose of this section is to describe a quandle 3-cocycle $\psi$ with $A=\check{\mathcal{P}} (\F_q)$; see Proposition \ref{prop7}.

First, we review another complex and a chain map defined in \cite{IK14}.
Let $Y$ be a set acted on by $\mathrm{SL}_2(\F_q)$. For example, if $Y$ is either $X_r $ or the projective space $\mathbb{P}^1(\F_q)=\F_q \cup \{ \infty\}$, then $\mathrm{SL}_2(\F_q)$ has a right action on $Y$.
Consider $C_n^{\Delta}(Y)$ to be the free $\Z$-module generated by $(n+1)$-tuples $(y_0, y_1, \dots, y_n) \in Y^{n+1} $, and define the differential map by
\vskip -0.55pc
\[ \partial_n^{\Delta} :C_n^{\Delta }(Y)\to C_{n-1}^{\Delta }(Y), \ \ \ \ \ \ \ \ \ ( y_0,\dots,y_n ) \mapsto \sum_{t=0}^n \ (-1)^t ( y_0,\dots, y_{t-1}, y_{t+1} ,\dots,y_n ) .\]
Let us define $C_*^{\Delta}(Y)_{\mathrm{SL}_2(\F_q)} $ to be the quotient complex $ C_*^{\Delta}(Y) \otimes_{\Z[\mathrm{SL}_2(\F_q)]}\Z .$
This quotient has been extensively studied, e.g., in \cite[Section 2]{Hut13}.

Next, we introduce a chain map $\varphi_* : C_*^R(X_r) \to C_*^{\Delta}(X_r)_{\mathrm{SL}_2(\F_q)}$ as follows.
Let $p_0$ be $(1,0) \in X_r$ and let $\varphi_0$ be the zero map. Furthermore, define $\varphi_n$ with $n \leq 3$ by
\begin{eqnarray*}\label{haa3i}
\varphi_2(a,b)&=&(p_0,a,b)-(p_0,a\lhd b ,b), \label{haa3wwi} \\
\nonumber \varphi_3(r,a,b)&=&(p_0,r,a,b)-(p_0,r\lhd a ,a,b)-(p_0,r\lhd b ,a\lhd b,b)+ (p_0,(r\lhd a) \lhd b ,a\lhd b,b) .
\end{eqnarray*}
Because we do not use any chain maps of higher degree, we omit the definition (see \cite{IK14} for details).
The map $\varphi_*$, which is often called the {\it Inoue-Kabaya chain map},  induces a homomorphism
$H_3^Q(X_r ) \to H_3^{\Delta}(X_r)_{\mathrm{SL}_2(\F_q)}$; see \cite[Theorem 3.4]{IK14}.
Thus, if we can find a suitable group $A$ and a 3-cocycle $\kappa$ of $C_*^{\Delta}(X_r)_{\mathrm{SL}_2(\F_q)} $,
we obtain a quandle 3-cocycle as the pullback $\varphi_3^*(\kappa)$.

Let us review the Bloch group of the finite field $\F_q $.
Let $\mathcal{P}(\F_q)$ be the abelian group presented by generators $[x]$ with $x \in \F_q \setminus \{ 0,1\}$ subject to the relations
\begin{equation*}\label{gokaku}
[x]-[y]+ [\frac{y}{x}] - [\frac{1-x^{-1}}{1-y^{-1}}] + [\frac{1-x}{1-y}]=0 , \ \ \ \ x \neq y \in \F_q \setminus \{ 0,1 \}. \end{equation*}
Let $\tilde{\wedge}^2 \F_q^{\times}$ denote the quotient of the multiplicative group $ \F_q^{\times}\otimes_{\Z} \F_q^{\times}$
by the subgroup generated by all $a \otimes b+b\otimes a$. It is easy to check $\tilde{\wedge}^2 \F_q^{\times} \cong \Z/2 $.
Consider the canonical homomorphism $\mathcal{P}(\F_q) \to \tilde{\wedge}^2 \F_q^{\times} $ that sends $[x]$ to $x \wedge (1-x)$.
The kernel is called \emph{the Bloch group}, and is denoted by $\mathcal{B}(\F_q)$; this kernel is isomorphic to $\Z/((q+1)/2)$ (see, e.g., \cite[Lemma 7.4]{Hut13}).
If $q \leq 200$, an explicit expression of the isomorphism $\mathcal{B}(\F_q) \cong \Z/((q+1)/2)$  can easily be obtained with the help of a computer program; see \cite[Appendix A]{Kar1} for $p=7,11,13$.
Moreover, let $\check{\mathcal{P}}(\F_q)$ be the quotient group of $\mathcal{P}(\F_q)$ given by the following relations:
\begin{equation*}
[x] = [1-\frac{1}{x}] = [\frac{1}{1-x}] =-[\frac{1}{x}]=-[\frac{x-1}{x}]=-[1-x], \qquad x \in \F_q \setminus \{ 0,1 \}.
\end{equation*}
Consider the projection $\mathcal{P}(\F_q)\to \check{\mathcal{P}}(\F_q)$ and denote the image of $\mathcal{B}(\F_q)$ by $\check{\mathcal{B}}(\F_q)$.
When $q$ is odd, the following isomorphisms hold (see \cite[Theorem 4.4]{Ohtuki} for details):
\begin{equation*}\label{reduced_bloch}
\check{\mathcal{B}}(\F_q) \cong \left\{
\begin{array}{lll}
\Z /(\frac{q+1}{2}) \Z, &\text{if } q \equiv 1, 9 \quad &\text{mod } 12,\\
\Z /(\frac{q+1}{4}) \Z, &\text{if } q \equiv 3, 7 \quad &\text{mod } 12,\\
\Z /(\frac{q+1}{6}) \Z, &\text{if } q \equiv 5 \quad &\text{mod } 12,\\
\Z /(\frac{q+1}{12}) \Z, &\text{if } q \equiv 11 \quad &\text{mod } 12.
\end{array}
\right.
\end{equation*}
In particular, if $q \equiv 1, 9 \text{ mod } 12$, we can naturally identify $\mathcal{B}(\F_q)$ with $\check{\mathcal{B}}(\F_q)$.

Next, we construct a specific homomorphism $H_3^{\Delta}(X_r )_{\SL_2(\F_q)} \to \check{\mathcal{P}}(\F_q)$.
For distinct points $x, y, z$ of $\mathbb{P}^1(\F_q)$, we define
\[ \phi(x,y,z):= \left\{
\begin{array}{ll}
(z-x)(x-y)(z-y)^{-1}, &\textrm{if } x,y,z \neq \infty, \\ (y-z)^{-1},&\textrm{if } x= \infty , \\ z-x,& \textrm{if } y= \infty , \\ x-y,& \textrm{if } z= \infty,
\end{array}
\right.\]
and define a homomorphism $\Phi: C_3^{\Delta}(\mathbb{P}^1(\F_q))_{\SL_2(\F_q)} \to \check{\mathcal{P}}(\F_q)$ by
\begin{equation*}
\Phi(x_1,x_2,x_3,x_4)=
\left\{
\begin{array}{ll}
\left[ \frac{\phi ( x_1,x_2,x_4)}{ \phi(x_1,x_2,x_3)} \right], &\textrm{if } x_i \neq x_j \textrm{ for all distinct } i \textrm{ and } j,\\
0, & \text{otherwise.}
\end{array}
\right.
\end{equation*}
Because $\Phi(\mathrm{Im}(\partial_4^{\Delta})) = \{ 0\}$ according to  \cite[Lemma 4.2]{Kar1} and $\mathrm{Im} \Phi \subset \check{\mathcal{B}} (\F_q)$, the map $\Phi$ induces a homomorphism $\Phi_* : H_3^\Delta (\mathbb{P}^1 (\F_q) )_{\SL_2(F_q)} \to \check{\mathcal{B}}(\F_q) $.
In summary, we have the following:
\begin{Proposition}\label{prop7}
Let $\mathrm{proj} : X_r \to \mathbb{P}^1(\F_q)$ be the canonical projection. The following composite map is a quandle 3-cocycle:
\[\psi_{ r } : H_3^Q(X_r) \stackrel{\varphi_3}{\longrightarrow} H_3^{\Delta}(X_r)_{\mathrm{SL}_2(\F_q)} \xrightarrow{\mathrm{proj}_*} H_3^{\Delta}(\mathbb{P}^1(\F_q))_{\mathrm{SL}_2(\F_q)} \stackrel{\Phi_*}{\longrightarrow} \check{ \mathcal{B}}(\F_q)\]
\end{Proposition}
In conclusion, because we can specifically describe $\check{\mathcal{P}}(\F_q)$ and $\check{\mathcal{B}}(\F_q)$ with generators,
we obtain a detailed presentation of the quandle 3-cocycle $\psi_r$.
This leads to the following:
\begin{Theorem}\label{prop888}
Suppose $q$ is a generic prime power.
Let $\mathrm{BW} : H_3 (\SL_2 (\F_q)) \to \mathcal{B} (\F_q)$ be the Bloch-Wigner map, and let $\widehat{\mathrm{BW}}$ be the composite of $\mathrm{BW}$ and the projection $\mathcal{B}(\F_q) \twoheadrightarrow \check{\mathcal{B}}(\F_q)$.
Then, there is an isomorphism $\sigma : \check{\mathcal{B}}(\F_q) \to \check{\mathcal{B}}(\F_q)$ such that
\[
\widehat{\mathrm{BW}}( \DW_q (L) )
=
\sum_{ \sh \in \col_{1}(D) } \sigma (\langle \psi_1, [\sh ]\rangle)
+
\sum_{ \sh \in \col_{r_0}(D) } \sigma (\langle \psi_{r_0}, [\sh ]\rangle) \in \Z[\check{\mathcal{B}}(\F_q)].
\]
\end{Theorem}

\subsection{Some computational examples with quandle cocycle invariant}
Using Theorem \ref{prop888} and a computer program, we can compute $\sum_{ \sh \in \col(D) } \langle \psi_r, [\sh ]\rangle \in \Z[ \check{\mathcal{B}} (\F_p)] $ for prime knots with seven or fewer crossings.
Table \ref{thepairing} presents the computational results.
Here, we do not consider twist knots because the resulting computations have  already been presented in \cite{Kar1},
and we omit the case with $p=11$ because $\check{\mathcal{B}} (\F_{11}) = 0$.
\begin{table}[htpb]
\begin{tabular}{|l||l|l|l|l|l|}\hline
&	$p=7$ & $p=13$ & $p=17$ & $p=19$ & $p=23$ \\ \hline
$6_2$	&	$96$ & $24$ & $544 t^2+576$ & $720$ & $1056$ \\ \hline
$6_3$	&	$12$ & $24$ & $544 t^2+544 t+32$ & $36$ & $2068$ \\ \hline
$7_3$	&	$12$ & $312 t^6+336$ & $32$ & $36$ & $2068$ \\ \hline
$7_4$	&	$168 t+12$	& $312 t^6+24$ & $544 t^2+544 t+576$ & $1368 t^2+36$ & $1056$ \\ \hline
$7_5$	&	$12$ & $24$ & $32$ & $36$ & $1056$ \\ \hline
$7_6$	&	$84 t+12$	& $24$ & $32$ & $684 t^2+36$ & $1056$ \\ \hline
$7_7$	&	$12$ & $336$ & $544 t+1120$ & $684 t^3+684 t+36$ & $3080$ \\ \hline\hline
$\check{\mathcal{B}} (\F_p)$ & $\left\langle t \mid t^{2} \right\rangle$ & $\left\langle t \mid t^{7} \right\rangle$ & $\left\langle t \mid t^{3} \right\rangle$ & $\left\langle t \mid t^{5} \right\rangle$ & $\left\langle t \mid t^{2} \right\rangle$ \\ \hline
\end{tabular}
\end{table}
\begin{table}[htpb]
\begin{tabular}{|l||l|l|l|}
\hline
& $p=29$ & $p=31$ & $p=37$ \\ \hline
$6_2$ & $1624 t^3+3248 t+56$ & $60$ & $72$ \\ \hline
$6_3$ & $1624 t^3+1624 t^2+56$ & $1860 t^{7}+1860 t+60$ & $72$ \\ \hline
$7_3$ & $56$ & $1860 t+1920$ & $2664 t^2+72$ \\ \hline
$7_4$ & $56$ & $1860 t^2+3720 t+60$ & $5328 t^7+2664 t^2+72$ \\ \hline
$7_5$ & $56$ & $60$ & $2664 t^{14}+2664 t^4+72$ \\ \hline
$7_6$ & $1624 t^3+1680$ & $1860 t^4+60$ & $2664 t^5+72$ \\ \hline
$7_7$ & $1624 t^4+56$ & $1860 t^4+60$ & $5328 t^{12}+72$ \\ \hline\hline
$\check{\mathcal{B}} (\F_p)$ & $\left\langle t \mid t^{5} \right\rangle$ & $\left\langle t \mid t^{8} \right\rangle$ & $\left\langle t \mid t^{19} \right\rangle$ \\ \hline
\end{tabular}
\caption{Sum of $2 \sum_{ \sh \in \col(D) } \langle \psi_r, [\sh ]\rangle /p(p-1) \in \Z[ \check{\mathcal{B}} (\F_p)]$ for prime knots with seven or fewer crossings.}\label{thepairing}
\end{table}
Moreover, using Propositions \ref{prop888} and \ref{prop8}, we can summarize the following:
\begin{Proposition}\label{redDW}\rm{(cf. \cite[Theorem 3.1]{Kar1})}.
Let $q=7$ and $D_m$ be a diagram of $\mathcal{K}_m$.
Then $\check{\mathcal{B}}(\F_q) \cong \langle t \mid t^2 = 1 \rangle$ and the reduced DW invariant of $\mathcal{K}_m$ is computed as
\begin{equation*}
\frac{1}{21}\sum_{ \sh \in \col(D_m) } \langle \psi_r, [\sh ]\rangle = \left\{
\begin{array}{lll}
12, &\text{if } \, m \equiv 0,2,6,8,12,14,18,20 \quad &\mathrm{mod } \,\, 24,\\
12+84 t, &\text{if } \, m \equiv 3,4,9,10,11,16,17,22 \quad &\mathrm{mod } \,\, 24,\\
12+168 t , &\text{if } \, m \equiv 1,19 \quad &\mathrm{mod } \,\, 24,\\
96, &\text{if } \, m \equiv 5,15,21,23 \quad &\mathrm{mod } \,\, 24,\\
96 + 84 t, &\text{if } \, m \equiv 7,13 \quad &\mathrm{mod } \,\, 24.
\end{array}
\right.
\end{equation*}
\end{Proposition}
\begin{proof}
It is sufficient to compute the left-hand side for $1 \leq m \leq q (q^2 -1)^2$ because of Corollary \ref{periodicity};
with the help of a computer program, we can easily show the required equality.
\end{proof}
Note that, for $q \leq 17$, we can compute the reduced DW invariant of $\mathcal{K}_m$ in the same manner.
\appendix
\def\thesection{\Alph{section}}

\section{Another definition of parabolic DW invariant}\label{app1}

Section \ref{sec2} defines the parabolic DW invariant with branched coverings.
Concerning 3-manifolds with boundaries, there are other approaches to defining 3-manifold invariants like the DW invariant using the relative group homology (see \cite{Nos, Zic, Wakui}). However,
it is not so easy to approach the (relative) fundamental class of 3-manifolds in terms of the relative group homology.

In the knot case,
we can briefly define the DW invariant with the relative group homology.
Let us review the relative group homology for groups $H \subset G$ and a commutative ring $A$.
Let $ C_{n}^{\rm gr }(G;A )$ be the free $A$-module in which the basis is formed by the elements $G^n ,$
and define the boundary map $\partial_n( g_1, \dots , g_n) \in C_{n-1}^{\mathrm{gr}}(G;A)$ by the formula
\[ ( g_2, \dots ,g_{n}) + \bigl( \sum_{ \ 1 \leq i \leq n-1}\!\! (-1)^i ( g_1, \dots ,g_{i-1}, g_{i} g_{i+1}, g_{i+2},\dots , g_n) \bigr) +(-1)^{n} ( g_1, \dots , g_{n-1}). \]
Furthermore, we define the relative complex by the mapping cone.
More precisely, $ C_{n}^{\rm gr }(G,H;A )$ is defined by $ C_{n}^{\rm gr }(G;A )\oplus C_{n-1}^{\rm gr }(H;A ) $, and the boundary map is defined by
\[ \partial_n \bigl(g_1, \dots, g_{n}, h_1, \dots, h_{n-1}) = \bigl( \partial_{n} ( g_1, \dots, g_{n}) + (-1)^n ( h_1, \dots, h_{n-1}), \ \partial_{n-1} ( h_1, \dots, h_{n-1}) \bigr).\]
Then, the relative homology is called {\it the relative group homology of $(G,H)$}, and is denoted by $H_n(G,H;A)$.
If a space pair $Y \subset X $ is an Eilenberg-MacLane space pair, the ordinary homology $H_n(X,Y;A)$ is
known to be isomorphic to $H_n(\pi_1(X),\pi_1(Y);A) $.

For example, consider the case where $X$ is a knot complement, and $Y=\partial X$ is the torus boundary.
Then, the integral third homology $ H_3(\pi_1(X),\pi_1(Y))$ is isomorphic to $\Z$, and the generator
corresponds with the fundamental class of $[X,\partial X]$.
In addition, consider a meridian $\mathfrak{m}$ and a preferred longitude $\mathfrak{l}$ in $\pi_1(X)$,
and suppose there is a homomorphism $f: \pi_1(S^3 \setminus K) \to G$ such that $\{ f (\mathfrak{m}),f(\mathfrak{l}) \}\subset H$, where $G$ is a group  and $H \subset G$ is a subgroup.
Then, we can canonically define the pushforward $f_* [X,\partial X] \in H_3(G, H)$ up to sign. Here, the sign depends on the choices of $( \mathfrak{m}, \mathfrak{l}) $.

To form a comparison with the parabolic DW invariants, we consider the special case of $G=\mathrm{SL}_2(\F_q)$ and take $H$ as the upper-triangular subgroup. 
Note that the group isomorphisms $H \cong \F_q \rtimes \Z/2 $ if $q$ is odd  and $H \cong \F_q $ if $q$ is even.
In particular, as a result of the transfer map (see \cite[\S 3.9]{Bro}), we can easily obtain $H_2(\F_q) \cong \F_q \wedge \F_q $.
In addition, when $q$ is generic, we have the long exact sequence
\begin{equation*}\label{seq6} H_3(H) \to H_3( G) \to H_3(G,H ) \stackrel{\delta}{\to} H_2(H) \to H_2(G ) =0 \ \ \ \ (\mathrm{exact}).\end{equation*}
Because $H_2(H)$ is annihilated by $q$ and $H_3( G) \cong \Z/(q^2 -1)$, we have the decomposition
\begin{equation}\label{seq7} H_3 (\mathrm{SL}_2(\F_q), H ) \cong H_3(\mathrm{SL}_2(\F_q)) \oplus H_2(H ) \cong \Z/(q^2 -1) \oplus (\F_q \wedge \F_q).\end{equation}
Denote the projection $H_3 (\mathrm{SL}_2(\F_q), H ;\Z ) \to \Z/(q^2 -1) $ by $\mathcal{P}$. 
Then, it is natural to consider the following definition:
\begin{equation}\label{def4}\mathrm{DW}'(K) = \sum_{f \in \Para(L,q) } 1_{\Z} \, \mathcal{P}(f_* [X,\partial X] ) \in \Z [\Z/(q^2 -1)] .\end{equation}
\begin{Proposition}\label{prop7}Let $L$ be a knot $K$, and suppose that $q$ is generic. With an appropriate choice of the projection $\mathcal{P}$, the definitions of the parabolic DW invariants are equal. That is, $\mathrm{DW}_q(K) =\mathrm{DW}'(K).$
\end{Proposition}
\begin{proof} Let $X$ be $S^3 \setminus K $.
Consider the inclusion $E_K^p \subset B_K^p$, which induces the isomorphisms
\[ H_3(B_K^p) \cong H_3(B_K^p, B_K^p \setminus E_K^p)\cong H_3 (E_K^p, \partial E_K^p) \] according to the homology long exact sequence and the excision axiom.
Denote the composite of the isomorphism by $\delta$. By definition, $\delta([B_p^K])= \pm [ E_K^p, \partial E_K^p ] $, leading to
$ \mathcal{P} ( \mathrm{res}(f)_* ([ E_K^p, \partial E_K^p ] ))= \overline{f}_* ( [B_p^K] )$ for some choice of $\mathcal{P}$.
Moreover, let $\mathrm{cov}: E_K^p \to S^3 \setminus K $ be the cyclic covering.
Then, $ \mathrm{cov}_* [ E_K^p , \partial E_K^p ] = p [X ,\partial X ] \in H_3(X,\partial X;\Z) \cong \Z$.
Thus, for any parabolic representation $f: \pi_1(S^3 \setminus K) \to \mathrm{SL}_2(\F_q)$,
\[ p \mathcal{P} (f_*( [X ,\partial X ] ) )= \mathcal{P} \circ \mathrm{res}(f)_* ([ E_K^p, \partial E_K^p ] )= \bar{f}_*([B_p^K]) \in H_3( \mathrm{SL}_2(\F_q) ) . \]
Replacing $ p \mathcal{P} $ by $\mathcal{P} $, we obtain the required equality from the definitions of the invariants.
\end{proof}
\begin{Remark}In \eqref{def4}, we omit the information on the summand $\F_q \wedge \F_q $ in \eqref{seq7}.
However, the following discussion demonstrates that this information is largely determined by the pair $( \mathfrak{m}, \mathfrak{l}) $.
Because $\pi_1( \partial X)=\Z \times \Z$, $ \mathfrak{m}$ and $\mathfrak{l} $ can be regarded as 1-cycles in the group homology $H_1( \Z \times \Z)$, and the cross product $ \mathfrak{m} \times \mathfrak{l} $ generates $H_2( \Z \times \Z) \cong \Z. $
Moreover, from the delta map of the long exact sequence, we can deduce that  $H_3( X,\partial X) \cong H_2(\partial X) =H_2( \Z \times \Z) .$
Thus, $ \delta ( f_*[X,\partial X] )= \pm f_*( \mathfrak{m} \times \mathfrak{l} )= \pm f_*( \mathfrak{m} ) \times f_*( \mathfrak{l})$.
Because the delta map $\delta$ can be regarded as the projection $ H_3 (\mathrm{SL}_2(\F_q), \F_q) \to \F_q \wedge \F_q $ in \eqref{seq7},
$ \delta ( f_*[X,\partial X] )$ is almost determined by the pair $( \mathfrak{m}, \mathfrak{l}) $ as required.
\end{Remark}

\subsection{Comparison with reduced DW invariant of Karuo}
Next, when $q$ is generic, Proposition \ref{prop8} states that our invariants in \eqref{def1} and \eqref{def4} is a lift of
the reduced DW invariant of Karuo \cite{Kar212, Kar1}.

First, we briefly review the reduced invariant.
Let $\nu K$ be an open tubular neighborhood of $K$ in $S^3$.
Because $\pi_1 (\partial (S^3 \setminus \nu K )) = \Z \times \Z$,
every essential simple closed curve $C$ on $\partial (S^3 \setminus \nu K )$ is presented by $(a, b) \in \Z \times \Z$, where $a$ and $b$ are coprime.
Let us call $C$ {\it the $(a, b)$-curve}. Let $M_{(a,b)}(K)$ be the closed 3-manifold obtained from $S^3 \setminus \nu K$ by the Dehn filling along $C$.
Then, for any parabolic representation $f : \pi_1(S^3 \setminus \nu K) \to \mathrm{SL}_2(\F_q) $,
the kernel of $f$ is non-trivial, because $ \pi_1(S^3 \setminus \nu K) $ is of infinite order and $ |\mathrm{SL}_2(\F_q) |< \infty$.
Thus, we can choose a coprime pair $(a,b ) \in \Z$ such that $ f ( C) $ is trivial; therefore, $f$ induces
$ f_{a,b}: \pi_1(M_{(a,b)}(K)) \to \mathrm{SL}_2(\F_q)$.
However, the pushforward $f_*([M_{(a,b)}(K)]) \in H_3(\mathrm{SL}_2(\F_q) )$ depends on the choice of $(a,b)$.

To obtain knot invariants, Karuo considered the Bloch group $\mathcal{B}(\F_q)$ and the Bloch-Wigner map $ {\mathrm{BW}} : H_3(\mathrm{SL}_2(\F_q) ) \to \mathcal{B}(\F_q)$; see \cite[Section 2]{Kar212} for the definitions.
Recall the groups $\mathcal{B}_q (\F_q)$ and $\check{\mathcal{B}}_q (\F_q)$ and the map $\widehat{\rm BW}$ from Section 3.2.
Karuo showed that the quotient $\widehat{\rm BW} \circ f_*([M_{(a,b)}(K)])$ in $\check{\mathcal{B}}(\F_q) $ is independent of the choice of $(a,b)$.
Then, {\it the reduced DW invariant} $\widehat{\rm DW}(K, \F_q) $ of $K$ is defined as the sum of the pushforwards $\widehat{\rm BW} \circ f_*([M_{(a,b)}(K)])$ over all parabolic representations $f$ in $\Z[ \check{\mathcal{B}}(\F_q) ]$.
Note that while the reduced DW invariant in \cite{Kar212, Kar1} is deﬁned as a formal sum of all conjugacy classes of parabolic representations with respect to the general linear group $\GL_2 (\F_q)$, this paper deﬁnes the reduced DW invariant as a formal sum of all parabolic representations.

Using a similar argument as in the proof of Proposition \ref{prop7}, we can show the following.
\begin{Proposition}\label{prop8}
Let $L$ be a knot $K$.
Suppose $q$ is generic. With an appropriate choice of the projection $\mathcal{P}$,
$ \widehat{\rm BW} (\mathrm{DW}_q(K))= \widehat{\rm DW}(K, \F_q) \in \Z[ \check{\mathcal{B}}(\F_q) ]$.
In particular, the reduced DW invariant is a specialization of the parabolic DW invariant $\DW_q (K)$ in \eqref{def1}.
\end{Proposition}

\section{Some computations of $\DW_{16} (L)$}\label{app2}
For $m \in \N$, let $L$ be either the $(2,2m)$-torus link $\mathcal{T}_m$ or the $2m$-twist knot $\mathcal{K}_m$.
Then, the double branched covering space $B_L^2$ is homeomorphic to a lens space.
Hence, we can apply the computational method described in Section \ref{commet} to $L$.
This appendix presents the computational results of $\DW^T_{16} (L)$ and $\DW^K_{16} (L)$ for $1\leq m \leq 50$.
Note that, for $50 < m$, the invariants can be easily calculated in a similar manner.
In this appendix, we let $\F_{16} \cong \F_2 [X]/(1 + X + X^4)$, and take $ \begin{pmatrix}X & 0 \\0 & X^{-1} \end{pmatrix}, \begin{pmatrix}0 & -1 \\1 & X \end{pmatrix} \in \SL_2 (\F_{16})$ as the generators of the subgroups $T,K \subset \SL_2 (\F_{16})$, respectively.
In addition, we identify $H_3 (T)$ with $\langle t \mid t^{15} =1 \rangle$ and $H_3 (K)$ with $\langle s \mid s^{17} =1 \rangle$.
\begin{table}[htpb]
\small
\begin{tabular}{l|c|c}
$m$ & $\DW^T_{16} (\mathcal{T}_m) / q(q+1)$ & $\DW^K_{16} (\mathcal{T}_m) / (q-1)q$ \\ \hline \hline
$1$ & $0$ & $0$ \\
$2$ & $0$ & $0$ \\
$3$ & $t^{10}$ & $0$ \\
$4$ & $0$ & $0$ \\
$5$ & $t^9+t^6$ & $0$ \\
$6$ & $t^5$ & $0$ \\
$7$ & $0$ & $0$ \\
$8$ & $0$ & $0$ \\
$9$ & $1$ & $0$ \\
$10$ & $t^{12}+t^3$ & $0$ \\
$11$ & $0$ & $0$ \\
$12$ & $t^{10}$ & $0$ \\
$13$ & $0$ & $0$ \\
$14$ & $0$ & $0$ \\
$15$ & $t^{12}+2 t^8+t^5+t^3+2 t^2$ & $0$ \\
$16$ & $0$ & $0$ \\
$17$ & $0$ & $s^{16}+s^{15}+s^{13}+s^9+s^8+s^4+s^2+s$ \\
$18$ & $1$ & $0$ \\
$19$ & $0$ & $0$ \\
$20$ & $t^9+t^6$ & $0$ \\
$21$ & $t^{10}$ & $0$ \\
$22$ & $0$ & $0$ \\
$23$ & $0$ & $0$ \\
$24$ & $t^5$ & $0$ \\
$25$ & $2$ & $0$ \\
$26$ & $0$ & $0$ \\
$27$ & $1$ & $0$ \\
$28$ & $0$ & $0$ \\
$29$ & $0$ & $0$ \\
$30$ & $t^{10}+t^9+t^6+2 t^4+2 t$ & $0$ \\
$31$ & $0$ & $0$ \\
$32$ & $0$ & $0$ \\
$33$ & $t^5$ & $0$ \\
$34$ & $0$ & $s^{16}+s^{15}+s^{13}+s^9+s^8+s^4+s^2+s$ \\
$35$ & $t^{12}+t^3$ & $0$ \\
$36$ & $1$ & $0$ \\
$37$ & $0$ & $0$ \\
$38$ & $0$ & $0$ \\
$39$ & $t^{10}$ & $0$ \\
$40$ & $t^{12}+t^3$ & $0$ \\
$41$ & $0$ & $0$ \\
$42$ & $t^5$ & $0$ \\
$43$ & $0$ & $0$ \\
$44$ & $0$ & $0$ \\
$45$ & $3 t^9+3 t^6+1$ & $0$ \\
$46$ & $0$ & $0$ \\
$47$ & $0$ & $0$ \\
$48$ & $t^{10}$ & $0$ \\
$49$ & $0$ & $0$ \\
$50$ & $2$ & $0$ \\
\end{tabular}
\caption{$\DW^T_{16} (\mathcal{T}_m) / q(q+1) \in \Z [H_3 (T)] $ and $\DW^K_{16} (\mathcal{T}_m) / (q-1)q \in \Z [H_3 (K)].$}
\label{paraDW_of_torus}
\end{table}
\clearpage
\begin{table}[htpb]
\small
\begin{tabular}{l|c|c}
$m$ & $\DW^T_{16} (\mathcal{K}_m) / q(q+1)$ & $\DW^K_{16} (\mathcal{K}_m) /(q-1)q$ \\ \hline \hline
$1$ & $t^9+t^6$ & $0$ \\
$2$ & $1$ & $0$ \\
$3$ & $0$ & $0$ \\
$4$ & $0$ & $s^{16}+s^{15}+s^{13}+s^9+s^8+s^4+s^2+s$ \\
$5$ & $t^{10}$ & $0$ \\
$6$ & $2$ & $0$ \\
$7$ & $0$ & $0$ \\
$8$ & $t^5$ & $0$ \\
$9$ & $0$ & $0$ \\
$10$ & $0$ & $0$ \\
$11$ & $3 t^9+3 t^6+1$ & $0$ \\
$12$ & $0$ & $0$ \\
$13$ & $0$ & $0$ \\
$14$ & $t^{10}$ & $0$ \\
$15$ & $0$ & $0$ \\
$16$ & $t^{12}+t^3$ & $0$ \\
$17$ & $t^5$ & $0$ \\
$18$ & $0$ & $0$ \\
$19$ & $0$ & $0$ \\
$20$ & $1$ & $0$ \\
$21$ & $t^{12}+t^3$ & $s^{14}+s^{12}+s^{11}+s^{10}+s^7+s^6+s^5+s^3$ \\
$22$ & $0$ & $0$ \\
$23$ & $t^{10}$ & $0$ \\
$24$ & $0$ & $0$ \\
$25$ & $0$ & $0$ \\
$26$ & $2 t^{14}+2 t^{11}+t^9+t^6+t^5$ & $0$ \\
$27$ & $0$ & $0$ \\
$28$ & $0$ & $0$ \\
$29$ & $1$ & $0$ \\
$30$ & $0$ & $0$ \\
$31$ & $2$ & $0$ \\
$32$ & $t^{10}$ & $0$ \\
$33$ & $0$ & $0$ \\
$34$ & $0$ & $0$ \\
$35$ & $t^5$ & $0$ \\
$36$ & $t^9+t^6$ & $0$ \\
$37$ & $0$ & $0$ \\
$38$ & $1$ & $s^{16}+s^{15}+s^{13}+s^9+s^8+s^4+s^2+s$ \\
$39$ & $0$ & $0$ \\
$40$ & $0$ & $0$ \\
$41$ & $2 t^{13}+t^{12}+t^{10}+2 t^7+t^3$ & $0$ \\
$42$ & $0$ & $0$ \\
$43$ & $0$ & $0$ \\
$44$ & $t^5$ & $0$ \\
$45$ & $0$ & $0$ \\
$46$ & $t^{12}+t^3$ & $0$ \\
$47$ & $1$ & $0$ \\
$48$ & $0$ & $0$ \\
$49$ & $0$ & $0$ \\
$50$ & $t^{10}$ & $0$ \\
\end{tabular}
\caption{$\DW^T_{16} (\mathcal{K}_m) / q(q+1) \in \Z [H_3 (T)] $ and $\DW^K_{16} (\mathcal{K}_m) / (q-1)q \in \Z [H_3 (K)].$}
\label{paraDW_of_twist}
\end{table}
\clearpage

\noindent{\bf Acknowledgments} 
I am grateful to my supervisor Takefumi Nosaka for many helpful suggestions. 
I also thank Hiroaki Karuo for valuable comments.

\bibliographystyle{amsalpha}
\nocite{*}
\bibliography{paraDW}

\providecommand{\bysame}{\leavevmode\hbox to3em{\hrulefill}\thinspace}
\providecommand{\MR}{\relax\ifhmode\unskip\space\fi MR }
\providecommand{\MRhref}[2]{%
  \href{http://www.ams.org/mathscinet-getitem?mr=#1}{#2}
}
\providecommand{\href}[2]{#2}
\begin{thebibliography}{Kar21b}

\bibitem[BE78]{BE78}
R.~Bieri and B.~Eckmann, \emph{Relative homology and {P}oincar\'{e} duality for
  group pairs}, J. Pure Appl. Algebra \textbf{13} (1978), no.~3, 277--319.

\bibitem[Bro94]{Bro}
K.~S. Brown, \emph{Cohomology of groups}, Graduate Texts in Mathematics,
  vol.~87, Springer-Verlag, New York, 1994.

\bibitem[CKS01]{CKS01}
J.~S. Carter, S.~Kamada, and M.~Saito, \emph{Geometric interpretations of
  quandle homology}, Journal of knot theory and its ramifications \textbf{10}
  (2001), no.~03, 345--386.

\bibitem[DW90]{DW}
R.~Dijkgraaf and E.~Witten, \emph{Topological gauge theories and group
  cohomology}, Comm. Math. Phys. \textbf{129} (1990), no.~2, 393--429.

\bibitem[Hut13]{Hut13}
K.~Hutchinson, \emph{A {B}loch-{W}igner complex for {$\mathrm{SL}_2$}}, J.
  K-Theory \textbf{12} (2013), no.~1, 15--68.

\bibitem[IK14]{IK14}
A.~Inoue and Y.~Kabaya, \emph{Quandle homology and complex volume}, Geometriae
  Dedicata \textbf{171} (2014), no.~1, 265--292.

\bibitem[Kar21a]{Kar212}
H.~Karuo, \emph{The reduced {D}ijkgraaf-{W}itten invariant of double twist
  knots in the {B}loch group of {$\Bbb F_p$}}, J. Knot Theory Ramifications
  \textbf{30} (2021), no.~7, Paper No. 2150055, 52.

\bibitem[Kar21b]{Kar1}
\bysame, \emph{The reduced {D}ijkgraaf-{W}itten invariant of twist knots in the
  {B}loch group of a finite field}, J. Knot Theory Ramifications \textbf{30}
  (2021), no.~3, Paper No. 2150014, 70.

\bibitem[Kim18]{Kim18}
Naoki Kimura, \emph{A generalization of the dijkgraaf-witten invariants for
  cusped 3-manifolds}, 2018, preprint, arXiv:1805.05130.

\bibitem[Nos15]{Nos15}
T.~Nosaka, \emph{Homotopical interpretation of link invariants from finite
  quandles}, Topology and its Applications \textbf{193} (2015), 1--30.

\bibitem[Nos17]{Nos17}
\bysame, \emph{Quandles and topological pairs}, SpringerBriefs in Mathematics,
  Springer, Singapore, 2017, Symmetry, knots, and cohomology.

\bibitem[Nos20]{Nos}
\bysame, \emph{On the fundamental 3-classes of knot group representations},
  Geom. Dedicata \textbf{204} (2020), 1--24.

\bibitem[Oht]{Ohtuki}
T.~Ohtsuki, \emph{On the bloch groups of finite fields and their quotients by
  the relation corresponding to a tetrahedral symmetry}, preprint,
  https://www.kurims.kyoto-u.ac.jp/preprint/file/RIMS1938.pdf.

\bibitem[Tak59]{Tak59}
S.~Takasu, \emph{Relative homology and relative cohomology theory of groups},
  J. Fac. Sci. Univ. Tokyo Sect. I \textbf{8} (1959), 75--110.

\bibitem[Wak92]{Wakui}
M.~Wakui, \emph{On {D}ijkgraaf-{W}itten invariant for {$3$}-manifolds}, Osaka
  J. Math. \textbf{29} (1992), no.~4, 675--696.

\bibitem[Zic09]{Zic}
C.~K. Zickert, \emph{The volume and {C}hern-{S}imons invariant of a
  representation}, Duke Math. J. \textbf{150} (2009), no.~3, 489--532.

\end{thebibliography}

\end{document}